\documentclass[11pt, final]{article}
\usepackage{a4}
\usepackage{amsmath}%
\usepackage{amstext}%
\usepackage{amssymb}%
\usepackage{showkeys}%
\usepackage{epsfig}%
\usepackage{cite}
\usepackage{tikz}
\usepackage{subfig}
\usepackage{caption}
\usepackage{multirow}
\usepackage{booktabs}
\usepackage{floatrow}

\usepackage{listings}
\newtheorem{theorem}{Theorem}

\newtheorem{axiom}{Axiom}

\newtheorem{corollary}[theorem]{Corollary}

\newtheorem{definition}[axiom]{Definition}

\newtheorem{lemma}[theorem]{Lemma}

\newtheorem{rem}[theorem]{Remark}
\newenvironment{remark}{\rem\rm}{\endrem}

\newcounter{unnumber}

\newtheorem{prf}[unnumber]{Proof.}
\newenvironment{proof}{\prf\rm}{\hfill{$\blacksquare$}\endprf}
\newcommand{\R}{\mathbb{R}}%
\newcommand{\N}{\mathbb{N}}%
\newcommand{\ol}{\overline}%

\renewcommand{\>}{\right\rangle}

\DeclareMathOperator*\dom{dom}%
\DeclareMathOperator*\prox{prox}%
\DeclareMathOperator*\argmin{argmin}

\DeclareMathOperator*\crit{crit}
\DeclareMathOperator*\dist{dist}
\DeclareMathOperator*\proj{proj}

\usepackage{geometry}
\title{Approaching nonsmooth nonconvex optimization problems through first order dynamical systems with hidden acceleration and Hessian driven damping terms}

\author{Radu Ioan Bo\c{t} \thanks{University of Vienna, Faculty of Mathematics, Oskar-Morgenstern-Platz 1, A-1090 Vienna, Austria,
email: radu.bot@univie.ac.at. Research partially supported by FWF (Austrian Science Fund), project I 2419-N32.} \and
Ern\"{o} Robert Csetnek \thanks {University of Vienna, Faculty of Mathematics, Oskar-Morgenstern-Platz 1, A-1090 Vienna, Austria,
email: ernoe.robert.csetnek@univie.ac.at. Research supported by FWF (Austrian Science Fund), Lise Meitner Programme, project M 1682-N25.}}

\begin{document}
\maketitle

\begin{center}
{\bf Dedicated to the memory of Jon Borwein, who was so inspiring and motivating}
\end{center}

\noindent \textbf{Abstract.} In this paper we carry out an asymptotic analysis of the proximal-gradient dynamical system \begin{equation*}\left\{
\begin{array}{ll}
\dot x(t) +x(t) = \prox_{\gamma f}\big[x(t)-\gamma\nabla\Phi(x(t))-ax(t)-by(t)\big],\\
\dot y(t)+ax(t)+by(t)=0
\end{array}\right.\end{equation*}
where $f$ is a proper, convex and lower semicontinuous function, $\Phi$ a possibly nonconvex  smooth function and $\gamma, a$ and $b$ are positive real numbers. We show that the generated trajectories approach the set of critical points of $f+\Phi$, here understood as zeros of its limiting subdifferential, under the premise that a regularization of this sum function satisfies the Kurdyka-\L{}ojasiewicz property. We also establish convergence rates for the trajectories, formulated in terms of the \L{}ojasiewicz exponent of the considered regularization function. 

\vspace{1ex}

\vspace{1ex}


\section{Introduction}\label{sec-intr}

We begin with a short literature review that serves as motivation for the research conducted in this paper.

The Newton-like dynamical system
\begin{equation}\label{dyn-syst-psi0-intr}
\ddot x(t) + \lambda\dot x(t) + \gamma\nabla^2\Phi(x(t))(\dot x(t)) +  \nabla \Phi(x(t))=0
\end{equation}
has been investigated by Alvarez, Attouch, Bolte and Redont in \cite{alv-att-bolte-red} in the context of asymptotically approaching the minimizers of the optimization problem
\begin{equation}\label{opt-intr-phi}\inf_{x\in \R^n}\Phi(x),\end{equation}
for $\Phi$ a smooth ${\cal C}^2$ function and $\lambda$ and $\gamma$ positive numbers. System \eqref{dyn-syst-psi0-intr} is a second order system both in time, due to the presence of the acceleration term $\ddot x(t)$, which is associated to inertial effects, and in space, due to presence of the Hessian $\nabla^2\Phi(x(t))$. 
The trajectories generated by \eqref{dyn-syst-psi0-intr} have been proved to converge to a critical point $\Phi$, when this function is analytic, and to a minimizer of $\Phi$, when it is convex. Dynamical systems of type \eqref{dyn-syst-psi0-intr} are of large interest, as they occur in different applications in fields like optimization, mechanics, control theory and PDE theory (see \cite{alv-att-bolte-red, alv-per1998, att-bolte-red2002, att-mainge-red2012, att-peyp-red2016, att-red2001}). 

The authors of \cite{alv-att-bolte-red} have also pointed out the surprising fact that the dynamical system \eqref{dyn-syst-psi0-intr} can be viewed as a first order dynamical system with no occurrence of the Hessian. More precisely, it has been shown that \eqref{dyn-syst-psi0-intr} is equivalent to 
\begin{equation}\label{dyn-syst-smooth}\left\{
\begin{array}{ll}
\dot x(t) +\gamma \nabla\Phi(x(t))+ax(t)+by(t)=0,\\
\dot y(t)+ax(t)+by(t)=0
\end{array}\right.\end{equation}
where $a:=\lambda - \frac{1}{\gamma}$ and $b:=\frac{1}{\gamma}$.  The obvious advantage of \eqref{dyn-syst-smooth} comes from the fact that for its asymptotic analysis no second order information on the smooth function $\Phi$ is needed. We refer to \cite{alv-att-bolte-red, att-peyp-red2016} for applications and other arguments in favor of this reformulation of \eqref{dyn-syst-psi0-intr}. 

On the other hand, in order to asmyptotically approach the minimizers of constrained optimization problems  of the form  
\begin{equation}\label{opt-intr-phi-C}\inf_{x\in C}\Phi(x),\end{equation}
where $C\subseteq \R^n$ is a nonempty, closed, convex set, the following projection-gradient dynamical system has been considered and investigated by Antipin \cite{antipin} and Bolte \cite{bolte-2003} 
\begin{equation}\label{intr-syst-ant-bolte}
\dot x(t)+x(t)=\proj\nolimits_C\big(x(t)-\gamma\nabla \Phi(x(t))\big).
\end{equation}
Here, $\proj_C : \R^n\rightarrow C$ denotes the \emph{projection operator} onto the set $C$.

These being given, the following combination of the systems  \eqref{dyn-syst-smooth} and \eqref{intr-syst-ant-bolte}
\begin{equation}\label{dyn-syst-proj}\left\{
\begin{array}{ll}
\dot x(t) +x(t) = \proj_C\big[x(t)-\gamma\nabla\Phi(x(t))-ax(t)-by(t)\big]\\
\dot y(t)+ax(t)+by(t)=0
\end{array}\right.\end{equation}
has been proposed in \cite{alv-att-bolte-red}, for $a, b$ and $\gamma$ positive numbers, in order to asymptotically approach the minimizers of the constrained optimization problem \eqref{opt-intr-phi-C} in the hypothesis that the objective function $\Phi$ is convex.

Proximal-gradient dynamical systems, which are generalizations of \eqref{intr-syst-ant-bolte}, 
have been recently considered by 
Abbas and Attouch in \cite[Section 5.2]{abbas-att-arx14} in the full convex setting. Implicit dynamical systems related to both optimization 
problems and monotone inclusions have been  considered in the literature also by Attouch and Svaiter in \cite{att-sv2011}, 
Attouch, Abbas and Svaiter in \cite{abbas-att-sv} and  Attouch, Alvarez and Svaiter in \cite{att-alv-sv}. These investigations have been 
continued and extended in \cite{bb-ts-cont, b-c-dyn-KM, b-c-dyn-pen, b-c-dyn-sec-ord, b-c-conv-rate-cont}. 

In the last years the interest in approaching the solvability of nonconvex optimization problems from continuous and discrete perspective is continuously increasing (see \cite{attouch-bolte2009, att-b-red-soub2010, att-b-sv2013, b-sab-teb, b-c-inertial-nonc-ts, bcl, c-pesquet-r, f-g-peyp, h-l-s-t, ipiano}). Following this tendency, we investigate in this paper the optimization problem 
\begin{equation}\label{opt-intr-f-phi}\inf_{x\in \R^n}\big(f(x)+\Phi(x)\big),\end{equation}
where $f$ is a (possibly nonsmooth) proper, convex and lower semicontinuous function and $\Phi$ a (possibly nonconvex) smooth function. More precisely, in this paper we investigate 
the convergence of the trajectories generated by the proximal-gradient dynamical system 
\begin{equation}\label{dyn}\left\{
\begin{array}{ll}
\dot x(t) +x(t) = \prox_{\gamma f}\big[x(t)-\gamma\nabla\Phi(x(t))-ax(t)-by(t)\big],\\
\dot y(t)+ax(t)+by(t)=0
\end{array}\right.\end{equation}
where $a,b$ and $\gamma$ are positive real numbers and 
$$\prox\nolimits_{\gamma f}:\R^n\rightarrow\R^n, \ \prox\nolimits_{\gamma f}(y)=\argmin_{u\in\R^n}\left\{f(u)+\frac{1}{2\gamma}\|u-y\|^2\right\},$$
denotes the \emph{proximal point operator} of $\gamma f$, to a critical point  of $f+\Phi$, here understood as a zero of its limiting subdifferential. To this end we assume that a regularization of the objective function satisfies the \emph{Kurdyka-\L{}ojasiewicz} property; in other words, it is a \emph{KL function}. The convergence analysis relies on methods and concepts of real algebraic geometry introduced by \L{}ojasiewicz \cite{lojasiewicz1963} and Kurdyka \cite{kurdyka1998} and later developed  in the nonsmooth setting by Attouch, Bolte and Svaiter \cite{att-b-sv2013} and Bolte, Sabach and Teboulle \cite{b-sab-teb}. 

In the convergence analyis we use three main ingredients: (1) we prove a Lyapunov-type property, expressed as a sufficient decrease of a regularization of the objective function along the trajectories, (2) we show the existence of a subgradient lower bound for the trajectories and,
finally, (3) we derive convergence by making use of the Kurdyka-\L{}ojasiewicz property of the objective function (for a similar approach in the continuous case see \cite{alv-att-bolte-red} and in the discrete setting see \cite{att-b-sv2013, b-sab-teb}). Furthermore, we obtain convergence rates for the trajectories expressed in terms of  the \L{}ojasiewicz exponent of the regularized objective function.

\section{Preliminaries}\label{sec2}

We recall  some notions and results which are needed throughout the paper. We consider on $\R^n$ the Euclidean scalar product and the corresponding norm denoted by $\langle\cdot,\cdot\rangle$ and $\|\cdot\|$, respectively. 

The {\it domain} of the function  $f:\R^n\rightarrow \R\cup\{+\infty\}$ is defined by $\dom f=\{x\in\R^n:f(x)<+\infty\}$. We say that $f$ is {\it proper}, if $\dom f\neq\emptyset$.  
For the following generalized subdifferential notions and their basic properties we refer to \cite{borwein-zhu, boris-carte, rock-wets}. 
Let $f:\R^n\rightarrow \R\cup\{+\infty\}$ be a proper and lower semicontinuous function. The {\it Fr\'{e}chet (viscosity)  
subdifferential} of $f$ at $x\in\dom f$ is the set $$\hat{\partial}f(x)= \left \{v\in\R^n: \liminf_{y\rightarrow x}\frac{f(y)-f(x)-\<v,y-x\>}{\|y-x\|}\geq 0 \right \}.$$ For 
$x\notin\dom f$, one sets $\hat{\partial}f(x):=\emptyset$. The {\it limiting (Mordukhovich) subdifferential} is defined at $x\in \dom f$ by 
$$\partial f(x)=\{v\in\R^n:\exists x_k\rightarrow x,f(x_k)\rightarrow f(x)\mbox{ and }\exists v_k\in\hat{\partial}f(x_k),v_k\rightarrow v \mbox{ as }k\rightarrow+\infty\},$$
while for $x \notin \dom f$, one takes $\partial f(x) :=\emptyset$. Therefore  $\hat\partial f(x)\subseteq\partial f(x)$ for each $x\in\R^n$.

When $f$ is convex, these subdifferential notions coincide with the {\it convex subdifferential}, thus 
$\hat\partial f(x)=\partial f(x)=\{v\in\R^n:f(y)\geq f(x)+\<v,y-x\> \ \forall y\in \R^n\}$ for all $x\in\R^n$. 

The following {\it closedness criterion} of the graph of the limiting subdifferential will be used in the convergence analysis: if $(x_k)_{k\in\N}$ and $(v_k)_{k\in\N}$ are sequences in $\R^n$ such that 
$v_k\in\partial f(x_k)$ for all $k\in\N$, $(x_k,v_k)\rightarrow (x,v)$ and $f(x_k)\rightarrow f(x)$ as $k\rightarrow+\infty$, then 
$v\in\partial f(x)$. 

The Fermat rule reads in this nonsmooth setting as follows: if $x\in\R^n$ is a local minimizer of $f$, then $0\in\partial f(x)$.  We denote by 
$$\crit(f)=\{x\in\R^n: 0\in\partial f(x)\}$$ the set of {\it (limiting)-critical points} of $f$. 

When $f$ is continuously differentiable around $x \in \R^n$ we have $\partial f(x)=\{\nabla f(x)\}$. We will also make use of the following subdifferential sum rule:
if $f:\R^n\rightarrow\R\cup\{+\infty\}$ is proper and lower semicontinuous  and $h:\R^n\rightarrow \R$ is a continuously differentiable function, then $\partial (f+h)(x)=\partial f(x)+\nabla h(x)$ for all $x\in\R^m$. 

A crucial role in the asymptotic analysis of the dynamical system \eqref{dyn} is played by the class of 
functions satisfying the {\it Kurdyka-\L{}ojasiewicz property}. For $\eta\in(0,+\infty]$, we denote by $\Theta_{\eta}$ the class of concave and continuous functions 
$\varphi:[0,\eta)\rightarrow [0,+\infty)$ such that $\varphi(0)=0$, $\varphi$ is continuously differentiable on $(0,\eta)$, continuous at $0$ and $\varphi'(s)>0$ for all 
$s\in(0, \eta)$. In the following definition (see \cite{att-b-red-soub2010, b-sab-teb}) we use also the {\it distance function} to a set, defined for $A\subseteq\R^n$ as $\dist(x,A)=\inf_{y\in A}\|x-y\|$  
for all $x\in\R^n$. 

\begin{definition}\label{KL-property} \rm({\it Kurdyka-\L{}ojasiewicz property}) Let $f:\R^n\rightarrow\R\cup\{+\infty\}$ be a proper and lower semicontinuous 
function. We say that $f$ satisfies the {\it Kurdyka-\L{}ojasiewicz (KL) property} at $\ol x\in \dom\partial f=\{x\in\R^n:\partial f(x)\neq\emptyset\}$, if there exist $\eta \in(0,+\infty]$, a neighborhood $U$ of $\ol x$ and a function $\varphi\in \Theta_{\eta}$ such that for all $x$ in the 
intersection 
$$U\cap \{x\in\R^n: f(\ol x)<f(x)<f(\ol x)+\eta\}$$ the following inequality holds 
$$\varphi'(f(x)-f(\ol x))\dist(0,\partial f(x))\geq 1.$$
If $f$ satisfies the KL property at each point in $\dom\partial f$, then $f$ is called {\it KL function}. 
\end{definition}

The origins of this notion go back to the pioneering work of \L{}ojasiewicz \cite{lojasiewicz1963}, where it is proved that for a real-analytic function 
$f:\R^n\rightarrow\R$ and a critical point $\ol x\in\R^n$ (that is $\nabla f(\ol x)=0$), there exists $\theta\in[1/2,1)$ such that the function 
$|f-f(\ol x)|^{\theta}\|\nabla f\|^{-1}$ is bounded around $\ol x$. This corresponds to the situation when $\varphi(s)=Cs^{1-\theta}$, where 
$C>0$. The result of 
\L{}ojasiewicz allows the interpretation of the KL property as a re-parametrization of the function values in order to avoid flatness around the 
critical points. Kurdyka \cite{kurdyka1998} extended this property to differentiable functions definable in o-minimal structures. 
Further extensions to the nonsmooth setting can be found in \cite{b-d-l2006, att-b-red-soub2010, b-d-l-s2007, b-d-l-m2010}. 

One of the remarkable properties of the KL functions is their ubiquity in applications (see \cite{b-sab-teb}). To the class of KL functions belong semi-algebraic, real sub-analytic, semiconvex, uniformly convex and 
convex functions satisfying a growth condition. We refer the reader to 
\cite{b-d-l2006, att-b-red-soub2010, b-d-l-m2010, b-sab-teb, b-d-l-s2007, att-b-sv2013, attouch-bolte2009} and the references therein for more on KL functions and illustrating examples. 

In the analysis below the following uniform KL property given in \cite[Lemma 6]{b-sab-teb} will be used.

\begin{lemma}\label{unif-KL-property} Let $\Omega\subseteq \R^n$ be a compact set and let $f:\R^n\rightarrow\R\cup\{+\infty\}$ be a proper 
and lower semicontinuous function. Assume that $f$ is constant on $\Omega$ and that it satisfies the KL property at each point of $\Omega$.   
Then there exist $\varepsilon,\eta >0$ and $\varphi\in \Theta_{\eta}$ such that for all $\ol x\in\Omega$ and all $x$ in the intersection 
\begin{equation}\label{int} \{x\in\R^n: \dist(x,\Omega)<\varepsilon\}\cap \{x\in\R^n: f(\ol x)<f(x)<f(\ol x)+\eta\}\end{equation} 
the inequality  \begin{equation}\label{KL-ineq}\varphi'(f(x)-f(\ol x))\dist(0,\partial f(x))\geq 1.\end{equation}
holds.
\end{lemma}

In the following we recall the notion of locally absolutely continuous function and state two of its basic properties. 

\begin{definition}\label{abs-cont} \rm (see, for instance, \cite{att-sv2011, abbas-att-sv}) 
A function $x : [0,+\infty) \rightarrow \R^n$ is said to be locally absolutely continuous, if it absolutely continuous on every interval $[0,T]$, where $T > 0$.
\end{definition}

\begin{remark}\label{rem-abs-cont}\rm\begin{enumerate} \item[(a)] An absolutely continuous function is differentiable almost 
everywhere, its derivative coincides with its distributional derivative almost everywhere and one can recover the function from its 
derivative $\dot x=y$ by integration. 

\item[(b)] If $x:[0,T]\rightarrow \R^n$ is absolutely continuous for $T > 0$ and $B:\R^n\rightarrow \R^n$ is 
$L$-Lipschitz continuous for $L\geq 0$, then the function $z=B\circ x$ is absolutely continuous, too.  Moreover, $z$ is differentiable almost everywhere on $[0,T]$ and the inequality 
$\|\dot z (t)\|\leq L\|\dot x(t)\|$ holds for almost every $t \in [0,T]$.  
\end{enumerate}
\end{remark}

The following two results, which can be interpreted as continuous versions of the quasi-Fej\'er monotonicity for sequences, 
will play an important role in the asymptotic analysis of the trajectories of the dynamical system investigated in this paper. 
For their proofs we refer the reader  to \cite[Lemma 5.1]{abbas-att-sv} and \cite[Lemma 5.2]{abbas-att-sv}, respectively.

\begin{lemma}\label{fejer-cont1} Suppose that $F:[0,+\infty)\rightarrow\R$ is locally absolutely continuous and bounded from below and that
there exists $G\in L^1([0,+\infty))$ such that for almost every $t \in [0,+\infty)$ $$\frac{d}{dt}F(t)\leq G(t).$$ 
Then there exists $\lim_{t\rightarrow \infty} F(t)\in\R$. 
\end{lemma}

\begin{lemma}\label{fejer-cont2}  If $1 \leq p < \infty$, $1 \leq r \leq \infty$, $F:[0,+\infty)\rightarrow[0,+\infty)$ is 
locally absolutely continuous, $F\in L^p([0,+\infty))$, $G:[0,+\infty)\rightarrow\R$, $G\in  L^r([0,+\infty))$ and 
for almost every $t \in [0,+\infty)$ $$\frac{d}{dt}F(t)\leq G(t),$$ then $\lim_{t\rightarrow +\infty} F(t)=0$. 
\end{lemma}

Further we recall a differentiability result that involves the composition of convex functions with absolutely 
continuous trajectories, which is due to Br\'{e}zis (\cite[Lemme 3.3, p. 73]{brezis}; see also \cite[Lemma 3.2]{att-cza-10}). 

\begin{lemma}\label{diff-brezis} Let $f:\R^n\rightarrow \R\cup\{+\infty\}$ be a proper, convex and lower semicontinuous function. 
Let $x\in L^2([0,T],\R^n)$ be absolutely continuous such that $\dot x\in L^2([0,T],\R^n)$ and $x(t)\in\dom f$ for almost every 
$t \in [0,T]$. Assume that there exists $\xi\in L^2([0,T],\R^n)$ such that $\xi(t)\in\partial f(x(t))$ for almost every $t \in [0,T]$. Then the function 
$t\mapsto f(x(t))$ is absolutely continuous and for almost every $t$ such that $x(t)\in\dom \partial f$ we have 
$$\frac{d}{dt}f(x(t))=\langle \dot x(t),h\rangle \ \forall h\in\partial f(x(t)).$$ 
\end{lemma}

We close this sesction with the following characterization of the proximal point operator of a proper, convex and lower semincontinuous function
$f:\R^n\rightarrow\R\cup\{+\infty\}$: for every $\gamma >0$ it holds (see for example \cite{bauschke-book})
\begin{equation}\label{ch-prox}p=\prox\nolimits_{\gamma f}(x) \mbox{ if and only if }x\in p+\gamma\partial f(p),\end{equation}
where $\partial f$ denotes the convex subdifferential of $f$.

\section{Asymptotic analysis}\label{sec3}

The dynamical system we investigate in this paper reads
\begin{equation}\label{dyn-syst}\left\{
\begin{array}{ll}
\dot x(t) +x(t) = \prox_{\gamma f}\big[x(t)-\gamma\nabla\Phi(x(t))-ax(t)-by(t)\big],\\
\dot y(t)+ax(t)+by(t)=0\\
x(0)=x_0, y(0)=y_0,
\end{array}\right.\end{equation}
where $x_0,y_0\in\R^n$ and  $a,b$ and $\gamma$ are positive real numbers. 
We assume that $f:\R^n\rightarrow\R\cup\{+\infty\}$ is proper, convex and lower semicontinuous, while 
$\Phi:\R^n\rightarrow\R$ is a Fr\'{e}chet differentiable with $L$-Lipschitz continuous gradient, for $L>0$, that is 
$\|\nabla\Phi(x)-\nabla\Phi(y)\|\leq L\|x-y\|$ for all $x,y\in \R^n$.

The existence and uniqueness of the trajectories generated by \eqref{dyn-syst} can be proved by using the estimates from the proof of Lemma \ref{l-decr} below and by following a classical argument, as in \cite[Theorem 7.1]{alv-att-bolte-red}. 

For the asymptotic analysis, we impose on the parameters involved the following condition: 
\begin{equation}\label{param}\left\{
\begin{array}{ll}
2\gamma L(|1-a|+\gamma L)+|1-a|+\gamma L+b\gamma L<1\\
ab+\frac{a}{2}+\frac{1}{2}a|1-a|+\frac{1}{2}\gamma aL+\frac{1}{2}\gamma abL <b
\end{array}\right.\end{equation}
and notice that the first inequality is fulfilled for an arbitrary $b>0$, if $a\in(0,2)$ and $\gamma>0$ are chosen small enough, while the second one holds for $a>0$ small enough. 

\subsection{Convergence of the trajectories}\label{subsec31}

We begin with the proof of a decrease property for a regularization of the objective function along the trajectories. 

\begin{lemma}\label{l-decr} Suppose that $f+\Phi$ is bounded from below and the parameters $a, b, \gamma$ and $L$ satisfy \eqref{param}. For $x_0,y_0\in\R^n$, let $(x,y) \in C^1([0,+\infty), \R^n)\times C^2([0,+\infty), \R^n)$ be 
the unique global solution of \eqref{dyn-syst}. Then the following statements are true: 
\begin{enumerate}
\item [(a)] $\frac{d}{dt}\left[(f+\Phi)(\dot x(t)+x(t))+\frac{1}{2\gamma}\|\dot x(t)\|^2+\frac{1}{2\gamma a}\|ax(t)+by(t)\|^2\right]
\leq -M_1\|\dot x(t)\|^2-M_2\|\dot y(t)\|^2$
for almost every $t\geq 0$, where $$M_1:=\frac{1}{2\gamma}- L(|1-a|+\gamma L)-\frac{1}{2\gamma}|1-a|-\frac{1}{2}L-\frac{1}{2}bL>0$$
and $$M_2:=\frac{b}{\gamma a} - \frac{b}{\gamma }-\frac{1}{2\gamma }-\frac{1}{2\gamma}|1-a|-\frac{1}{2}L-\frac{1}{2}bL>0;$$
\item [(b)]$\dot x,\dot y, ax+by\in L^2([0,+\infty);\R^n)$ and 
$\lim_{t\rightarrow+\infty}\dot x(t)=\lim_{t\rightarrow+\infty}\dot y(t)= \lim_{t\rightarrow+\infty}( ax(t)+by(t))=0$;
\item [(c)] $\exists\lim_{t\rightarrow+\infty}(f+\Phi)\big(\dot x(t)+x(t)\big)\in\R$.
\end{enumerate}
\end{lemma}

\begin{proof} Define $z:[0,+\infty)\rightarrow \R^n$ by 
\begin{equation}\label{z}z(t)=\prox\nolimits_{\gamma f}\big[x(t)-\gamma\nabla\Phi(x(t))-ax(t)-by(t)\big].\end{equation}
Since $\prox_{\gamma f}$ is nonexpansive (that is $1$-Lipschitz continuous), in view of Remark \ref{rem-abs-cont}(b), $z$ is locally absolutely continuous. From the Lipschitz continuity of 
$\nabla\Phi$ we obtain 
$$\|z(t)-z(s)\|\leq (|1-a|+\gamma L)\|x(t)-x(s)\|+b\| y(t)-y(s)\| \ \forall t,s\geq 0,$$
hence, for almost every $t\geq 0$, \begin{equation}\label{norm-dz}\|\dot z(t)\|\leq (|1-a|+\gamma L)\|\dot x(t)\|+b\| \dot y(t)\|.\end{equation}

Since \begin{equation}\label{x-z}\dot x(t)+x(t)=z(t) \ \forall t\geq 0,\end{equation} it follows that 
$\dot x$ is locally absolutely continuous, hence $\ddot x$ exists almost everywhere on $[0,+\infty)$
and for almost every $t\geq 0$ it holds
\begin{equation}\label{norm-ddx}\|\ddot x(t)\|\leq (1+|1-a|+\gamma L)\|\dot x(t)\|+b\| \dot y(t)\|.\end{equation}

We fix an arbitrary $T>0$. From the characterization \eqref{ch-prox} of the proximal point operator we have 
\begin{equation}\label{from-def-prox}-\frac{1}{\gamma}\dot x(t)-\frac{a}{\gamma}x(t)-\frac{b}{\gamma}y(t)-\nabla \Phi(x(t))
\in\partial f(\dot x(t)+x(t)) \ \forall t \in [0,+\infty).\end{equation}
Due to the continuity properties of the trajectories and their derivatives on $[0,T]$, \eqref{norm-ddx} and the Lipschitz continuity of $\nabla \Phi$, we have 
$$x, \dot x, \dot y, \ddot x, \nabla \Phi(x) \in L^2([0,T];\R^n).$$
Applying Lemma \ref{diff-brezis} we obtain that the function $t\mapsto f\big(\dot x(t)+x(t)\big)$ is absolutely continuous 
and $$\frac{d}{dt}f\big(\dot x(t)+x(t)\big)=\left\langle -\frac{1}{\gamma}\dot x(t)-\frac{a}{\gamma}x(t)-\frac{b}{\gamma}y(t)-\nabla \Phi(x(t)),\ddot x(t)+\dot x(t)\right\rangle$$
for almost every $t \in [0,T]$. Moreover, it holds 
$$\frac{d}{dt}\Phi\big(\dot x(t)+x(t)\big)=\left\langle \nabla \Phi\big(\dot x(t)+x(t)\big),\ddot x(t)+\dot x(t)\right\rangle$$
for almost every $t \in [0,T]$. Summing up the last two equalities and by taking into account \eqref{dyn-syst}, we obtain
\begin{align}\label{d1}\frac{d}{dt}(f+\Phi)\big(\dot x(t)+x(t)\big) = 
 & -\frac{1}{2\gamma}\frac{d}{dt}\big(\|\dot x(t)\|^2\big)-\frac{1}{\gamma}\|\dot x(t)\|^2
 -\frac{1}{\gamma}\langle ax(t)+by(t),\ddot x(t)+\dot x(t)\rangle\nonumber\\
 & +\left\langle \nabla \Phi\big(\dot x(t)+x(t)\big)-\nabla \Phi(x(t)), 
\ddot x(t)+\dot x(t)\right\rangle\nonumber\\
=  & -\frac{1}{2\gamma}\frac{d}{dt}\big(\|\dot x(t)\|^2\big)-\frac{1}{\gamma}\|\dot x(t)\|^2
 +\frac{1}{\gamma}\langle \dot y(t),\ddot x(t)+\dot x(t)\rangle\\
 & +\left\langle \nabla \Phi\big(\dot x(t)+x(t)\big)-\nabla \Phi(x(t)), 
\ddot x(t)+\dot x(t)\right\rangle\nonumber
\end{align}
for almost every $t\in[0,T]$. Further, due to \eqref{dyn-syst} we have 
\begin{align*}\frac{d}{dt}\left(\frac{1}{2}\|ax(t)+by(t)\|^2\right)=
& \langle ax(t)+by(t),a\dot x(t)+b\dot y(t)\rangle\\
= & -a\langle \dot x(t),\dot y(t)\rangle-b\|\dot y(t)\|^2. 
\end{align*}
Substituting the term $\langle \dot x(t),\dot y(t)\rangle$ from the last relation into \eqref{d1} we get 
\begin{align*}\frac{d}{dt}(f+\Phi)\big(\dot x(t)+x(t)\big)  
=  & -\frac{1}{2\gamma}\frac{d}{dt}\big(\|\dot x(t)\|^2\big)-\frac{1}{\gamma}\|\dot x(t)\|^2\\
& -\frac{1}{\gamma a}\frac{d}{dt}\left(\frac{1}{2}\|ax(t)+by(t)\|^2\right)-\frac{b}{\gamma a}\|\dot y(t)\|^2\\
 & +\frac{1}{\gamma}\langle \dot y(t),\ddot x(t)\rangle +\left\langle \nabla \Phi\big(\dot x(t)+x(t)\big)-\nabla \Phi(x(t)), 
\ddot x(t)+\dot x(t)\right\rangle\nonumber\\
\leq &  -\frac{1}{2\gamma}\frac{d}{dt}\big(\|\dot x(t)\|^2\big)-\frac{1}{\gamma}\|\dot x(t)\|^2-\frac{b}{\gamma a}\|\dot y(t)\|^2\\
& -\frac{1}{\gamma a}\frac{d}{dt}\left(\frac{1}{2}\|ax(t)+by(t)\|^2\right)
+\frac{1}{\gamma}(1+|1-a|+\gamma L)\|\dot x(t)\|\cdot\|\dot y(t)\|\nonumber\\
& + \frac{b}{\gamma}\|\dot y(t)\|^2+L\|\dot x(t)\|\cdot\|\ddot x(t)+\dot x(t)\|
\end{align*}
for almost every $t \in [0,T]$.
Noticing that $$\|\ddot x(t)+\dot x(t)\|=\|\dot z(t)\|$$ and by taking into account \eqref{norm-dz}, we derive 
\begin{align*}\frac{d}{dt}(f+\Phi)\big(\dot x(t)+x(t)\big)  
\leq &  -\frac{1}{2\gamma}\frac{d}{dt}\big(\|\dot x(t)\|^2\big)-\frac{1}{\gamma a}\frac{d}{dt}\left(\frac{1}{2}\|ax(t)+by(t)\|^2\right)\\
& -\left(\frac{1}{\gamma}-L\big(|1-a|+\gamma L\big)\right)\|\dot x(t)\|^2-\left(\frac{b}{\gamma a}-\frac{b}{\gamma}\right)\|\dot y(t)\|^2\\
& + \frac{1}{\gamma}\left(1+|1-a|+\gamma L+\gamma bL\right)\|\dot x(t)\|\cdot \|\dot y(t)\|
\end{align*}
for almost every $t \in [0,T]$. Finally, by using the inequality $\|\dot x(t)\|\cdot \|\dot y(t)\|\leq \frac{1}{2}\|\dot x(t)\|^2+\frac{1}{2}\|\dot y(t)\|^2$ and by taking into account the definitions of $M_1$ and $M_2$, we conclude that (a) holds. 

(b) By integration we get
\begin{align}\label{integ}
& (f+\Phi)\big(\dot x(T)+x(T)\big)+\frac{1}{2\gamma}\|\dot x(T)\|^2+ \frac{1}{2\gamma a}\|a x(T)+by(T)\|^2  + 
M_1 \int_{0}^T \|\dot x(t)\|^2dt\nonumber \\&  +M_2 \int_{0}^T \|\dot y(t)\|^2dt\leq 
(f+\Phi)\big(\dot x(0)+x_0\big)+\frac{1}{2\gamma}\|\dot x(0)\|^2+\frac{1}{2\gamma a}\|a x_0+b y_0\|^2. 
\end{align}
Since $f+\Phi$ is bounded from below and by taking into account that $T > 0$ has been arbitrarily chosen, we obtain
\begin{equation}\label{dot-l2}\dot x,\dot y\in L^2([0,+\infty);\R^n).      \end{equation}
Due to \eqref{norm-ddx}, this further implies
\begin{equation}\label{ddot-l2}
\ddot x\in L^2([0,+\infty);\R^n).      
\end{equation}
Furthermore, for almost every $t \in [0,+\infty)$ we have 
$$\frac{d}{dt}\big(\|\dot x(t)\|^2\big)=2\langle \dot x(t),\ddot x(t)\rangle\leq \|\dot x(t)\|^2+\|\ddot x(t)\|^2.$$
By applying Lemma \ref{fejer-cont2}, it follows that $\lim_{t\rightarrow+\infty}\dot x(t)=0$. 
Moreover, from \eqref{dyn-syst} we get that $\ddot y$ exists and $\ddot y\in L^2([0,+\infty);\R^n)$ due to \eqref{dot-l2}. 
The same arguments are used in order to conclude $\lim_{t\rightarrow+\infty}\dot y(t)=0$. 

(c) From (a) we get 
$$\frac{d}{dt}\left[(f+\Phi)(\dot x(t)+x(t))+\frac{1}{2\gamma}\|\dot x(t)\|^2+\frac{1}{2\gamma a}\|ax(t)+by(t)\|^2\right]\leq 0$$
for almost every $t \geq 0$. From Lemma \ref{fejer-cont1} it follows that
$$\lim_{t\rightarrow+\infty}\left[(f+\Phi)(\dot x(t)+x(t))+\frac{1}{2\gamma}\|\dot x(t)\|^2+\frac{1}{2\gamma a}\|ax(t)+by(t)\|^2\right]$$
exists and it is a real number, hence from 
$$\lim_{t\rightarrow+\infty}\dot x(t)=\lim_{t\rightarrow+\infty}\dot y(t)= \lim_{t\rightarrow+\infty}(-a x(t)-by(t))=0$$ the conclusion follows. 

\end{proof}

We define the {\it limit set of $x$} as 
$$\omega (x)=\{\ol x\in\R^n:\exists t_k\rightarrow+\infty \mbox{ such that }x(t_k)\rightarrow\ol x \mbox{ as }k\rightarrow+\infty\}.$$

\begin{lemma}\label{l-lim-crit-f} Suppose that $f+\Phi$ is bounded from below and the parameters $a, b, \gamma$ and $L$ satisfy \eqref{param}. For $x_0,y_0\in\R^n$, let $(x,y) \in C^1([0,+\infty), \R^n)\times C^2([0,+\infty), \R^n)$ be 
the unique global solution of \eqref{dyn-syst}. Then $$\omega(x)\subseteq \crit (f+\Phi).$$ 
\end{lemma}

\begin{proof} Let $\ol x\in\omega (x)$ and $t_k\rightarrow+\infty \mbox{ be such that }x(t_k)\rightarrow\ol x 
\mbox{ as }k\rightarrow+\infty.$ From \eqref{from-def-prox} we have 
\begin{align} & \ -\frac{1}{\gamma}\dot x(t_k)-\frac{a}{\gamma}x(t_k)-\frac{b}{\gamma}y(t_k)-\nabla \Phi(x(t_k))
+\nabla \Phi\big(\dot x(t_k)+x(t_k)\big) \nonumber\\
\in & \ \partial f\big(\dot x(t_k)+x(t_k)\big)+
\nabla \Phi\big(\dot x(t_k)+x(t_k)\big)\label{incl-tk} =  \ \partial (f+\Phi)\big(\dot x(t_k)+x(t_k)\big) \ \forall k \in \N.
\end{align}
Lemma \ref{l-decr}(b), \eqref{dyn-syst} and the Lipschitz continuity of $\nabla \Phi$ ensure that 
\begin{equation}\label{bor1} -\frac{1}{\gamma}\dot x(t_k)-\frac{a}{\gamma}x(t_k)-\frac{b}{\gamma}y(t_k)-\nabla \Phi(x(t_k))
+\nabla \Phi\big(\dot x(t_k)+x(t_k)\big)\rightarrow 0 \mbox{ as }k\rightarrow+\infty 
\end{equation}
and 
\begin{equation}\label{bor2} \dot x(t_k)+x(t_k)\rightarrow \ol x \mbox{ as }k\rightarrow+\infty. 
\end{equation}

We claim that \begin{equation}\label{bor3} \lim_{k\rightarrow+\infty}(f+\Phi)\big(\dot x(t_k)+x(t_k)\big)=(f+\Phi)(\ol x).\end{equation}
Indeed, from \eqref{bor2} and the lower semicontinuity of $f$ we get 
\begin{equation}\label{from-f-lsc}\liminf_{k\rightarrow+\infty}f\big(\dot x(t_k)+x(t_k)\big)\geq f(\ol x).\end{equation}
Further, since 
\begin{align*}
& \dot x(t_k)+x(t_k)= \argmin_{u\in\R^n}\left[f(u)+
\frac{1}{2\gamma}\left\|u-\big(x(t_k)-\gamma\nabla \Phi(x(t_k))-ax(t_k)-by(t_k)\big)\right\|^2\right]\\
& = \argmin_{u\in\R^n}\left[f(u)+\frac{1}{2\gamma}\|u-\big(x(t_k)-ax(t_k)-by(t_k)\big)\|^2+\langle u-\big(x(t_k)-ax(t_k)-by(t_k)\big),\nabla \Phi(x(t_k))\rangle\right],
\end{align*}
we have the inequality 
\begin{align*}
& \ f(\dot x(t_k)+x(t_k))+\frac{1}{2\gamma}\|\dot x(t_k)-ax(t_k)-by(t_k)\|^2+\langle \dot x(t_k)-ax(t_k)-by(t_k),\nabla\Phi(x(t_k))\rangle\\
\leq & \ f(\ol x)+\frac{1}{2\gamma}\|\ol x-\big(x(t_k)-ax(t_k)-by(t_k)\big)\|^2+
\langle \ol x-\big(x(t_k)-ax(t_k)-by(t_k)\big),\nabla\Phi(x(t_k))\rangle \ \forall k\in\N.
\end{align*}
Taking in the above inequality the limit as $k\rightarrow+\infty$, we derive by using again Lemma \ref{l-decr}(b) that 
\begin{equation*}\limsup_{k\rightarrow+\infty}f\big(\dot x(t_k)+x(t_k)\big)\leq f(\ol x),\end{equation*}
which combined with \eqref{from-f-lsc} implies 
\begin{equation*}\lim_{k\rightarrow+\infty}f\big(\dot x(t_k)+x(t_k)\big)= f(\ol x).\end{equation*}
By using \eqref{bor2}  and the continuity of $\Phi$ we conclude that \eqref{bor3} is true. 

Altogether, from \eqref{incl-tk}, \eqref{bor1}, \eqref{bor2}, \eqref{bor3} and the closedness criteria of the limiting subdifferential we 
obtain $0\in\partial (f+\Phi)(\ol x)$ and the proof is complete.
\end{proof}

\begin{lemma}\label{l-h123} Suppose that $f+\Phi$ is bounded from below and the parameters $a, b, \gamma$ and $L$ satisfy \eqref{param}. For $x_0,y_0\in\R^n$, let $(x,y) \in C^1([0,+\infty), \R^n)\times C^2([0,+\infty), \R^n)$ be 
the unique global solution of \eqref{dyn-syst}. Consider the function
$$H:\R^n\times\R^n\times\R^n\to\R\cup\{+\infty\},\, H(u,v,w)=(f+\Phi)(u)+\frac{1}{2\gamma}\|u-v\|^2+\frac{1}{2\gamma a}\|av+bw\|^2.$$
Then the following statements are true:
\begin{itemize}
\item[($H_1$)] for almost every $t\in [0,+\infty)$ it holds 
$$\frac{d}{dt}H\big(\dot x(t)+x(t),x(t),y(t)\big)\leq -M_1\|\dot x(t)\|^2-M_2\|\dot y(t)\|^2\leq 0$$ and 
$$\exists\lim_{t\rightarrow +\infty}H\big(\dot x(t)+x(t),x(t),y(t)\big)\in\R;$$
\item[($H_2$)] when $\zeta:[0,+\infty)\rightarrow \R^n\times\R^n\times\R^n$ is defined by  
$$\zeta(t):=\left(-\nabla \Phi(x(t))+\nabla \Phi\big(\dot x(t)+x(t)\big)+\frac{1}{\gamma}\dot y(t),
-\frac{1}{\gamma}\dot x(t)-\frac{1}{\gamma}\dot y(t),-\frac{b}{\gamma a}\dot y(t)\right),$$
then for every $t\in [0,+\infty)$ it holds $$\zeta(t)\in\partial H\big(\dot x(t)+x(t),x(t),y(t)\big)$$ 
and $$\|\zeta(t)\|\leq \left(\frac{2}{\gamma}+\frac{b}{\gamma a}\right)\|\dot y(t)\|+\left(L+\frac{1}{\gamma}\right)\|\dot x(t)\|;$$ 
\item[($H_3$)] for $\ol x\in\omega (x)$ and $t_k\rightarrow+\infty$ such that $x(t_k)\rightarrow\ol x$ as $k\rightarrow+\infty$, it holds
$$H\big(\dot x(t_k)+x(t_k),x(t_k),y(t_k)\big)\rightarrow (f+\Phi)(\ol x)=H\left(\ol x,\ol x,-\frac{a}{b}\ol x\right) \ \mbox{as} \ k\rightarrow+\infty.$$
\end{itemize}
\end{lemma}

\begin{proof} (H1) follows from Lemma \ref{l-decr}. The first statement in (H2) is a consequence of \eqref{from-def-prox}, the equation 
$\dot y(t)+ax(t)+by(t)=0$ and the fact that
\begin{equation}\label{H-subdiff}\partial H(u,v,w)=\left(\partial (f+\Phi)(u)+\frac{1}{\gamma}(u-v)\right)\times
\left\{\frac{1}{\gamma}(v-u)+\frac{1}{\gamma}(av+bw)\right\}\times\left\{\frac{b}{\gamma a}(av+bw)\right\}\end{equation}
for all $(u,v,w)\in\R^n\times\R^n\times\R^n$. 
The second statement in (H2) is a consequence of the Lipschitz continuity of $\nabla \Phi$. Finally, (H3) has been shown as intermediate 
step in the proof of Lemma \ref{l-lim-crit-f}. 
\end{proof}

\begin{lemma}\label{l} Suppose that $f+\Phi$ is bounded from below and the parameters $a, b, \gamma$ and $L$ satisfy \eqref{param}. For $x_0,y_0\in\R^n$, let $(x,y) \in C^1([0,+\infty), \R^n)\times C^2([0,+\infty), \R^n)$ be 
the unique global solution of \eqref{dyn-syst}. Consider the function
$$H:\R^n\times\R^n\times\R^n\to\R\cup\{+\infty\},\, H(u,v,w)=(f+\Phi)(u)+\frac{1}{2\gamma}\|u-v\|^2+\frac{1}{2\gamma a}\|av+bw\|^2.$$
Suppose that  $x$ is bounded. Then the following statements are true:
\begin{itemize}
\item[(a)] $\omega(\dot x+x,x,y)\subseteq \crit(H)=\{\left(u,u,-\frac{a}{b}u\right)\in\R^n\times\R^n\times\R^n:u\in \crit(f+\Phi)\}$; 
\item[(b)] $\lim_{t\to+\infty}\dist\Big(\big(\dot x(t)+x(t),x(t),y(t)\big),\omega\big(\dot x + x,x,y\big)\Big)=0$;
\item[(c)] $\omega\big(\dot x+x,x,y\big)$ is nonempty, compact and connected;
\item[(d)] $H$ is finite and constant on $\omega\big(\dot x+x,x,y\big).$
\end{itemize}
\end{lemma}

\begin{proof} (a), (b) and (d) are direct consequences of Lemma \ref{l-decr}, Lemma \ref{l-lim-crit-f} and Lemma \ref{l-h123}.

Finally, (c) is a classical result from \cite{haraux}. We also refer the reader  to the proof of Theorem 4.1 in 
\cite{alv-att-bolte-red}, where it is shown that the properties of $\omega(x)$ of being nonempty, compact and connected 
are generic for bounded trajectories fulfilling  $\lim_{t\rightarrow+\infty}{\dot x(t)}=0$. 
\end{proof}

\begin{remark}\label{cond-x-bound} 
Suppose that $a,b, \gamma$ and $L>0$ fulfill the inequality \eqref{param} and $f+\Phi$ is coercive, in other words,
$$\lim_{\|u\|\rightarrow+\infty}(f+\Phi)(u)=+\infty.$$ 
For $x_0,y_0\in\R^n$, let $(x,y) \in C^1([0,+\infty), \R^n)\times C^2([0,+\infty), \R^n)$ be 
the unique global solution of \eqref{dyn-syst}. Then $f+\Phi$ is bounded from below and  
$x$ is bounded.  

Indeed, since $f+\Phi$ is a proper, lower semicontinuous and coercive function, it follows that 
$\inf_{u\in\R^n}[f(u)+\Phi(u)]$ is finite and the infimum is attained. Hence $f+\Phi$ is bounded from below. On the other hand, from \eqref{integ} it follows
$$(f+\Phi)\big(\dot x(T)+x(T)\big)\leq 
  (f+\Phi)\big(\dot x(0)+x_0\big)+\frac{1}{2\gamma}\|\dot x(0)\|^2+\frac{1}{2\gamma a}\|a x_0+b y_0\|^2  \ \forall T \geq 0.$$
Since $f+\Phi$ is coercive, the lower level sets of $f+\Phi$ are bounded, hence the above inequality yields that $\dot x+x$ is bounded, which 
combined with $\lim_{t\rightarrow+\infty}\dot x(t)=0$ delivers the boundedness of $x$. Notice that in this case 
$y$ is bounded, too, due to Lemma \ref{l-decr}(b) and the equation $\dot y(t)+ax(t)+by(t)=0$.
\end{remark}

Now we are in the position to present the first main result of the paper, which concerns the convergence of the trajectories generated by \eqref{dyn-syst}. 

\begin{theorem}\label{conv-kl} Suppose that $f+\Phi$ is bounded from below and the parameters $a, b, \gamma$ and $L$ satisfy \eqref{param}.  For $x_0,y_0\in\R^n$, let $(x,y) \in C^1([0,+\infty), \R^n)\times C^2([0,+\infty), \R^n)$ be 
the unique global solution of \eqref{dyn-syst}. Consider the function
$$H:\R^n\times\R^n\times\R^n\to\R\cup\{+\infty\},\, H(u,v,w)=(f+\Phi)(u)+\frac{1}{2\gamma}\|u-v\|^2+\frac{1}{2\gamma a}\|av+bw\|^2.$$
Suppose that  $x$ is bounded. Then the following statements are true:
\begin{itemize}\item[(a)] $\dot x,\dot y,ax+by\in L^1([0,+\infty);\R^n)$ and 
$\lim_{t\rightarrow+\infty}\dot x(t)=\lim_{t\rightarrow+\infty}\dot y(t)= \lim_{t\rightarrow+\infty}( ax(t)+by(t))=0$;
\item[(b)] there exists $\ol x\in\crit(f+\Phi)$ such that $\lim_{t\rightarrow+\infty}x(t)=\ol x$ and 
$\lim_{t\rightarrow+\infty}y(t)=-\frac{a}{b}\ol x$.
\end{itemize}
\end{theorem}

\begin{proof} According to Lemma \ref{l}, we can choose an element $\ol x\in\crit (f+\Phi)$ such that 
$\left(\ol x,\ol x,-\frac{a}{b}\ol x\right)\in \omega (\dot x+x,x,y)$. According to Lemma \ref{l-h123}, it follows that
$$\lim_{t\rightarrow+\infty}H\big(\dot x(t)+x(t),x(t),y(t)\big)=H\left(\ol x,\ol x,-\frac{a}{b}\ol x\right).$$

We consider the following two cases.  

I. There exists $\ol t\geq 0$ such that $$H\big(\dot x(\ol t)+x(\ol t),x(\ol t),y(\ol t)\big)=H\left(\ol x,\ol x,-\frac{a}{b}\ol x\right).$$ 
Since from Lemma \ref{l-h123}(H1) we have $$\frac{d}{dt}H\big(\dot x(t)+x(t),x(t),y(t)\big) \leq 0 \ \forall t \in [0,+\infty),$$ 
we obtain for every $t\geq \ol t$ that
$$H\big(\dot x(t)+x(t),x(t),y(t)\big)\leq H\big(\dot x(\ol t)+x(\ol t),x(\ol t),y(\ol t)\big)=H\left(\ol x,\ol x,-\frac{a}{b}\ol x\right).$$ 
Thus $H\big(\dot x(t)+x(t),x(t),y(t)\big)=H\left(\ol x,\ol x,-\frac{a}{b}\ol x\right)$ for every $t\geq \ol t$. According to Lemma \ref{l-h123}(H1), it follows that $\dot x(t)=\dot y(t)=0$ for almost every $t \in [\ol t, +\infty)$, hence $x$ and $y$ are constant on 
$[\ol t,+\infty)$ and the conclusion follows. 

II. For every $t\geq 0$ it holds $H\big(\dot x(t)+x(t),x(t),y(t)\big)>H\left(\ol x,\ol x,-\frac{a}{b}\ol x\right).$ Take $\Omega:=\omega(\dot x+x,x,y)$. 

By using Lemma \ref{l}(c) and (d) and the fact that $H$ is a KL function, by Lemma \ref{unif-KL-property}, there exist positive numbers $\epsilon$ and $\eta$ and 
a concave function $\varphi\in\Theta_{\eta}$ such that for all
\begin{align}\label{int-H} 
(u,v,w)\in & \{(u,v,w)\in\R^n\times\R^n\times\R^n: \dist((u,v,w),\Omega)<\epsilon\} \nonumber \\ 
 & \cap\left\{(u,v,w)\in\R^n\times\R^n\times\R^n:H\left(\ol x,\ol x,-\frac{a}{b}\ol x\right)<H(u,v,w)<H\left(\ol x,\ol x,-\frac{a}{b}\ol x\right)+\eta\right\},\end{align}
one has
\begin{equation}\label{ineq-H}\varphi'\left(H(u,v,w)-H\left(\ol x,\ol x,-\frac{a}{b}\ol x\right)\right)\dist\Big((0,0,0),\partial H(u,v,w)\Big)\ge 1.\end{equation}

Let $t_1\geq 0$ be such that $H\big(\dot x(t)+x(t),x(t),y(t)\big)<H\left(\ol x,\ol x,-\frac{a}{b}\ol x\right)+\eta$ for all $t\geq t_1$. Since 
$\lim_{t\to+\infty}\dist\Big(\big(\dot x(t)+x(t),x(t),y(t)\big),\Omega\Big)=0$, there exists $t_2\geq 0$ such that for all $t\geq t_2$ the inequality 
$\dist\Big(\big(\dot x(t)+x(t),x(t),y(t)\big),\Omega\Big)<\epsilon$ holds. Hence for all $t\geq T:=\max\{t_1,t_2\}$, 
$\big(\dot x(t)+x(t),x(t),y(t)\big)$ belongs to the intersection in \eqref{int-H}. Thus, according to \eqref{ineq-H}, for every $t\geq T$ we have
\begin{equation}\label{ineq-Ht1}\varphi'\left(H\big(\dot x(t)+x(t),x(t),y(t)\big)-H\left(\ol x,\ol x,-\frac{a}{b}\ol x\right)\right)
\dist\Big((0,0,0),\partial H\big(\dot x(t)+x(t),x(t),y(t)\big)\Big)\ge 1.\end{equation}
By applying Lemma \ref{l-h123}(H2) we obtain for almost every $t \in [T, +\infty)$
\begin{equation}\label{ineq-Ht2}\Big(C_1\|\dot x(t)\|+C_2\|\dot y(t)\|\Big)\varphi'\left(H\big(\dot x(t)+x(t),x(t),y(t)\big)-H\left(\ol x,\ol x,-\frac{a}{b}\ol x\right)\right)
\ge 1,\end{equation}
where $$C_1:=L+\frac{1}{\gamma}\mbox{ and }C_2:= \frac{2}{\gamma}+\frac{b}{\gamma a}.$$
From here, by using Lemma \ref{l-h123}(H1), that $\varphi'>0$ and 
\begin{align*}
& \frac{d}{dt}\varphi\left(H\big(\dot x(t)+x(t),x(t),y(t)\big)-H\left(\ol x,\ol x,-\frac{a}{b}\ol x\right)\right)=\\
& \varphi'\left(H\big(\dot x(t)+x(t),x(t),y(t)\big)-H\left(\ol x,\ol x,-\frac{a}{b}\ol x\right)\right)\frac{d}{dt}H\big(\dot x(t)+x(t),x(t),y(t)\big),
\end{align*}
we deduce that for almost every $t \in [T, +\infty)$ it holds
\begin{equation}\label{ineq-pt-conv-r} \frac{d}{dt}\varphi\left(H\big(\dot x(t)+x(t),x(t),y(t)\big)-H\left(\ol x,\ol x,-\frac{a}{b}\ol x\right)\right)\leq 
-\frac{M_1\|\dot x(t)\|^2+M_2\|\dot y(t)\|^2}{C_1\|\dot x(t)\|+C_2\|\dot y(t)\|}.\end{equation}
Let be $\alpha>0$ (which does not depend on $t$) such that 
$$-\frac{M_1\|\dot x(t)\|^2+M_2\|\dot y(t)\|^2}{C_1\|\dot x(t)\|+C_2\|\dot y(t)\|}\leq -\alpha\|\dot x(t)\|-\alpha\|\dot y(t)\| \ \forall t\geq 0.$$
From \eqref{ineq-pt-conv-r} we derive the inequality 
\begin{equation}\label{ineq-pt-conv-r2} \frac{d}{dt}\varphi\left(H\big(\dot x(t)+x(t),x(t),y(t)\big)-H\left(\ol x,\ol x,-\frac{a}{b}\ol x\right)\right)\leq 
-\alpha\|\dot x(t)\|-\alpha\|\dot y(t)\|,\end{equation}
which holds for almost every $t\geq T$. 
Since $\varphi$ is bounded from below, by integration it follows  $\dot x,\dot y\in L^1([0,+\infty);\R^n)$. 
From here we obtain that $\lim_{t\rightarrow+\infty}x(t)$ exists and the conclusion follows from the results 
obtained in this section.
\end{proof}

Since the class of semi-algebraic functions is closed under addition (see for example \cite{b-sab-teb}) and 
$(u,v,w) \mapsto c\|u-v\|^2+c'\|av+bw\|^2$ is semi-algebraic for $c,c'>0$, we obtain the following direct 
consequence of the above theorem.  

\begin{corollary}\label{conv-semi-alg} Suppose that $f+\Phi$ is bounded from below and the parameters $a, b, \gamma$ and $L$ satisfy \eqref{param}.  For $x_0,y_0\in\R^n$, let $(x,y) \in C^1([0,+\infty), \R^n)\times C^2([0,+\infty), \R^n)$ be 
the unique global solution of \eqref{dyn-syst}. 
Suppose that  $x$ is bounded and $f+\Phi$ is semi-algebraic. Then the following statements are true:
\begin{itemize}\item[(a)] $\dot x,\dot y,ax+by\in L^1([0,+\infty);\R^n)$ and 
$\lim_{t\rightarrow+\infty}\dot x(t)=\lim_{t\rightarrow+\infty}\dot y(t)= \lim_{t\rightarrow+\infty}( ax(t)+by(t))=0$;
\item[(b)] there exists $\ol x\in\crit(f+\Phi)$ such that $\lim_{t\rightarrow+\infty}x(t)=\ol x$ and 
$\lim_{t\rightarrow+\infty}y(t)=-\frac{a}{b}\ol x$.
\end{itemize}
\end{corollary}

\subsection{Convergence rates}\label{subsec32}

In this subsection we investigate the convergence rates of the trajectories generated by the dynamical system \eqref{dyn-syst}. 
When solving optimization problems involving KL functions, convergence rates have been proved to depend on the so-called  \L{}ojasiewicz exponent 
(see \cite{lojasiewicz1963, b-d-l2006, attouch-bolte2009, f-g-peyp}). The main result of this subsection refers to the KL functions 
which satisfy Definition \ref{KL-property}  for $\varphi(s)=Cs^{1-\theta}$, where $C>0$ and $\theta\in(0,1)$. We recall the following 
definition considered in \cite{attouch-bolte2009}. 

\begin{definition}\label{kl-phi} \rm Let $f:\R^n\rightarrow\R\cup\{+\infty\}$ be a proper and lower semicontinuous function. 
The function $f$ is said to have the \L{}ojasiewicz property, if for every $\ol x\in\crit f$ there exist $C,\varepsilon >0$ and 
$\theta\in(0,1)$ such that 
\begin{equation}\label{kl-phi-ineq}|f(x)-f(\ol x)|^{\theta}\leq C\|x^*\| \ \mbox{for every} \ x \ \mbox{fulfilling} \ \|x-\ol x\|<\varepsilon \mbox{ and every} \ x^*\in\partial f(x).\end{equation}
\end{definition}

According to \cite[Lemma 2.1 and Remark 3.2(b)]{att-b-red-soub2010}, the KL property is automatically 
satisfied at any noncritical point, fact which motivates the restriction to critical points in the above definition. The real number $\theta$ in the above definition is called \emph{\L{}ojasiewicz exponent} of the function $f$ at the critical point  $\ol x$. 

\begin{theorem}\label{conv-r} Suppose that $f+\Phi$ is bounded from below and the parameters $a, b, \gamma$ and $L$ satisfy \eqref{param}.  For $x_0,y_0\in\R^n$, let $(x,y) \in C^1([0,+\infty), \R^n)\times C^2([0,+\infty), \R^n)$ be 
the unique global solution of \eqref{dyn-syst}. Consider the function
$$H:\R^n\times\R^n\times\R^n\to\R\cup\{+\infty\},\, H(u,v,w)=(f+\Phi)(u)+\frac{1}{2\gamma}\|u-v\|^2+\frac{1}{2\gamma a}\|av+bw\|^2.$$
Suppose that  $x$ is bounded and $H$ satisfies Definition \ref{KL-property} for $\varphi(s)=Cs^{1-\theta}$, where $C>0$ and 
$\theta\in(0,1)$. Then there exists $\ol x\in\crit (f+\Phi)$ such that 
$\lim_{t\rightarrow+\infty}x(t)=\ol x$ and $\lim_{t\rightarrow+\infty}y(t)=-\frac{a}{b}\ol x$. Let $\theta$ be the \L{}ojasiewicz 
exponent of $H$ at $\left(\ol x,\ol x,-\frac{a}{b}\ol x\right)\in\crit H$, according to the Definition \ref{kl-phi}. Then there exist 
$a_1,b_1,a_2,b_2>0$ and $t_0\geq 0$ such that for every $t\geq t_0$ the following statements are true: 
\begin{itemize}\item[(a)] if $\theta\in (0,\frac{1}{2})$, then $x$ and $y$ converge in finite time;
\item[(b)] if $\theta=\frac{1}{2}$, then $\|x(t)-\ol x\|+\|y(t)+\frac{a}{b}\ol x\|\leq a_1\exp(-b_1t)$;
\item[(c)] if $\theta\in (\frac{1}{2},1)$, then $\|x(t)-\ol x\|+\|y(t)+\frac{a}{b}\ol x\|\leq (a_2t+b_2)^{-\left(\frac{1-\theta}{2\theta-1}\right)}$.
\end{itemize}
\end{theorem}

\begin{proof} We define $\sigma:[0,+\infty)\rightarrow[0,+\infty)$ by (see also \cite{b-d-l2006})  
$$\sigma(t)=\int_{t}^{+\infty}\|\dot x(s)\|ds +\int_{t}^{+\infty}\|\dot y(s)\|ds \ \mbox{ for all }t\geq 0.$$
It is immediate that 
\begin{equation}\label{x-sigma1}\|x(t)-\ol x\|\leq \int_{t}^{+\infty}\|\dot x(s)\|ds \ \forall t\geq 0.\end{equation}

Indeed, this follows by noticing that for $T \geq t$
\begin{align*}\|x(t)-\ol x\|= \ & \|x(T)-\ol x-\int_t^T\dot x(s)ds\|\\
                          \leq  & \ \|x(T)-\ol x\|+\int_{t}^T\|\dot x(s)\|ds,
                          \end{align*}
and by letting afterwards $T\rightarrow +\infty$.

Similarly we have 
\begin{equation}\label{x-sigma2}\left\|y(t)+\frac{a}{b}\ol x\right\|\leq \int_{t}^{+\infty}\|\dot y(s)\|ds \ \forall t\geq 0.\end{equation}
From \eqref{x-sigma1} and \eqref{x-sigma2} we derive 
\begin{equation}\label{x-sigma}\|x(t)-\ol x\|+\left\|y(t)+\frac{a}{b}\ol x\right\|\leq \sigma(t)\ \forall t\geq 0.\end{equation}

We assume that for every $t\geq 0$ we have $H\left(\dot x(t)+x(t),x(t),y(t)\right)>H\left(\ol x,\ol x,-\frac{a}{b}\ol x\right).$ As seen in the proof of 
Theorem \ref{conv-kl}, in the other case the conclusion follows automatically. Furthermore, by invoking again the proof 
of above-named result, there exist $t_0\geq 0$ and $\alpha>0$ such that for almost every $t\geq t_0$ (see \eqref{ineq-pt-conv-r})
\begin{equation*} \alpha \|\dot x(t)\|+\alpha \|\dot y(t)\|+ \frac{d}{dt}\left[\left(H\big(\dot x(t)+x(t),x(t),y(t)\right)-
H\left(\ol x,\ol x,-\frac{a}{b}\ol x\right)\right]^{1-\theta}\leq 0\end{equation*}
and
\begin{equation*} \left\| \big(\dot x(t)+x(t),x(t),y(t)\big)-\left(\ol x,\ol x,-\frac{a}{b}\ol x\right)\right\|<\varepsilon.\end{equation*}

We derive by integration (for $T\geq t\geq t_0$) 
$$\alpha\int_t^{T}\|\dot x(s)\|ds+\alpha\int_t^{T}\|\dot y(s)\|ds+\left[\left(H\big(\dot x(T)+x(T),x(T),y(T)\big)-H\left(\ol x,\ol x,-\frac{a}{b}\ol x\right)\right)\right]^{1-\theta} $$$$
\leq\left[\left(H\big(\dot x(t)+x(t),x(t),y(t)\big)-H\left(\ol x,\ol x,-\frac{a}{b}\ol x\right)\right)\right]^{1-\theta},$$
hence 
\begin{equation}\label{ineq1}\alpha\sigma (t)\leq \left[\left(H\big(\dot x(t)+x(t),x(t),y(t)\big)-
H\left(\ol x,\ol x,-\frac{a}{b}\ol x\right)\right)\right]^{1-\theta} \ \forall t\geq t_0.\end{equation}

Since $\theta$ is the \L{}ojasiewicz exponent of $H$ at $\left(\ol x,\ol x,-\frac{a}{b}\ol x\right)$, we have 
$$\left|H\left(\dot x(t)+x(t),x(t),y(t)\right)-H\left(\ol x,\ol x,-\frac{a}{b}\ol x\right)\right|^{\theta}\leq C\|x^*\|
\ \forall x^*\in \partial H\left(\dot x(t)+x(t),x(t),y(t)\right)$$
for every $t\geq t_0$. 
According to Lemma \ref{l-h123}(H2), we can find $x^*(t)\in \partial H\left(\dot x(t)+x(t),x(t),y(t)\right)$ and a constant $N>0$ 
such that for every $t \in [0, +\infty)$ $$\|x^*(t)\|\leq N\|\dot x(t)\|+N\|\dot y(t)\|.$$
From the above two inequalities we derive for almost every $t \in [t_0, +\infty)$
$$\left|H\left(\dot x(t)+x(t),x(t),y(t)\right)-H\left(\ol x,\ol x,-\frac{a}{b}\ol x\right)\right|^{\theta}\leq C\cdot N\|\dot x(t)\|+C\cdot N\|\dot y(t)\|,$$ 
which combined with \eqref{ineq1} yields 
\begin{equation}\label{ineq2}\alpha\sigma (t)\leq \big(C\cdot N\|\dot x(t)\|+C\cdot N\|\dot y(t)\|\big)^{\frac{1-\theta}{\theta}}.\end{equation}

Since \begin{equation}\label{dsigma}\dot \sigma (t)=-\|\dot  x(t)\|-\|\dot  y(t)\|\end{equation} we conclude that there exists $\alpha'>0$ such that for almost every $t \in [t_0, +\infty)$
\begin{equation}\label{sigma} \dot\sigma (t)\leq -\alpha'\big(\sigma(t)\big)^{\frac{\theta}{1-\theta}}. 
\end{equation}

If $\theta=\frac{1}{2}$, then $$\dot\sigma (t)\leq -\alpha'\sigma(t)$$ for almost every $t \in [t_0, +\infty)$. By multiplying with 
$\exp(\alpha' t)$ and integrating afterwards 
from $t_0$ to $t$, it follows that there exist $a_1,b_1>0$ such that 
$$\sigma (t)\leq a_1\exp(-b_1t) \ \forall t\geq t_0$$ and the conclusion of (b) is immediate from \eqref{x-sigma}. 

Assume that $0<\theta<\frac{1}{2}$. We obtain from \eqref{sigma} 
$$\frac{d}{dt}\left(\sigma(t)^{\frac{1-2\theta}{1-\theta}}\right)\leq-\alpha' \frac{1-2\theta}{1-\theta}$$
for almost every $t \in [t_0, +\infty)$.

By integration we get $$\sigma(t)^{\frac{1-2\theta}{1-\theta}}\leq -\ol \alpha t+\ol \beta \ \forall t\geq t_0,$$ 
where $\ol \alpha>0$. Thus there exists $T\geq 0$ such that $$\sigma (T)\leq 0 \  \forall t\geq T,$$
which implies that $x$ and $y$ are constant on $[T,+\infty)$. 

Finally, suppose that $\frac{1}{2}<\theta<1$. We obtain from \eqref{sigma} 
$$\frac{d}{dt}\left(\sigma(t)^{\frac{1-2\theta}{1-\theta}}\right)\geq\alpha' \frac{2\theta-1}{1-\theta}$$
for almost every $t \in [t_0, +\infty)$. By integration one derives $$\sigma(t)\leq (a_2t+b_2)^{-\left(\frac{1-\theta}{2\theta-1}\right)} \ \forall t\geq t_0,$$ where 
$a_2,b_2>0$. Statement (c) follows from \eqref{x-sigma}.
\end{proof}


\begin{thebibliography}{99}

\bibitem{abbas-att-arx14} B. Abbas, H. Attouch, {\it Dynamical systems and forward-backward algorithms associated 
with the sum of a convex subdifferential and a monotone cocoercive operator}, Optimization 64(10), 2223--2252, 2015

\bibitem{abbas-att-sv} B. Abbas, H. Attouch, B.F. Svaiter, {\it Newton-like dynamics and forward-backward methods for 
structured monotone inclusions in Hilbert spaces}, Journal of Optimization Theory and its Applications 161(2), 331--360, 2014

\bibitem{alvarez2000} F. Alvarez, {\it On the minimizing property of a second order dissipative system in Hilbert spaces}, SIAM Journal
on Control and Optimization 38(4), 1102--1119, 2000

\bibitem{alvarez2004} F. Alvarez, {\it Weak convergence of a relaxed and inertial hybrid projection-proximal point algorithm for maximal monotone operators in Hilbert space}, SIAM Journal on Optimization 14(3), 773--782, 2004

\bibitem{alv-att-sva} F. Alvarez, H. Attouch, {\it An inertial proximal method for maximal monotone operators via discretization of  a nonlinear oscillator with damping}, Set-Valued Analysis 9(1-2), 3–11, 2001

\bibitem{alv-att-bolte-red} F. Alvarez, H. Attouch, J. Bolte, P. Redont, {\it A second-order gradient-like dissipative dynamical system 
with Hessian-driven damping. Application to optimization and mechanics}, Journal de Math\'{e}matiques Pures et Appliqu\'{e}es (9) 81(8), 
747--779, 2002

\bibitem{alv-per1998} F. Alvarez, J.M. P\'{e}rez, {\it A dynamical system associated with Newton's method for parametric approximations
of convex minimization problems}, Applied Mathematics and Optimization 38(2), 193--217, 1998

\bibitem{antipin} A.S. Antipin, {\it Minimization of convex functions on convex sets by means of differential equations}, 
(Russian) Differentsial'nye Uravneniya 30(9), 1475--1486, 1994; translation in Differential Equations 30(9), 1365--1375, 1994

\bibitem{att-alv} H. Attouch, F. Alvarez, {\it The heavy ball with friction dynamical system for convex constrained 
minimization problems}, in: Optimization (Namur, 1998), 25--35, in: Lecture Notes in Economics and Mathematical Systems 481, Springer, Berlin, 2000

\bibitem{abm} H. Attouch, G. Buttazzo, G. Michaille, {\it Variational Analysis in Sobolev and BV Spaces: Applications to PDEs and Optimization, Second Edition},
MOS-SIAM Series on Optimization, Philadelphia, 2014

\bibitem{att-alv-sv} H. Attouch, M. Marques Alves, B.F. Svaiter, {\it A dynamic approach to a proximal-Newton method for monotone inclusions 
in Hilbert spaces, with complexity $O(1/{n^2})$}, Journal of Convex Analyis 23(1), 139--180, 2016   

\bibitem{attouch-bolte2009} H. Attouch, J. Bolte, {\it On the convergence of the proximal algorithm for nonsmooth functions involving analytic
features}, Mathematical Programming 116(1-2) Series B, 5--16, 2009

\bibitem{att-bolte-red2002} H. Attouch, J. Bolte, P. Redont, {\it Optimizing properties of an inertial dynamical system with
geometric damping. Link with proximal methods}, Control and Cybernetics 31(3), 643--657, 2002

\bibitem{att-b-red-soub2010} H. Attouch, J. Bolte, P. Redont, A. Soubeyran, {\it Proximal alternating minimization and projection
methods for nonconvex problems: an approach based on the Kurdyka-\L{}ojasiewicz inequality}, Mathematics of Operations Research 
35(2), 438--457, 2010

\bibitem{att-b-sv2013} H. Attouch, J. Bolte, B.F. Svaiter, {\it Convergence of descent methods for semi-algebraic and tame problems: 
proximal algorithms, forward-backward splitting, and regularized Gauss-Seidel methods}, Mathematical Programming 137(1-2) Series A, 91--129, 2013

\bibitem{att-cza-10} H. Attouch, M.-O. Czarnecki, {\it Asymptotic behavior of coupled dynamical systems with multiscale aspects}, 
Journal of Differential Equations 248(6), 1315--1344, 2010

\bibitem{att-g-r} H. Attouch, X. Goudou, P. Redont, {\it The heavy ball with friction method. I. The continuous dynamical system: 
global exploration of the local minima of a real-valued function by asymptotic analysis of a dissipative dynamical system}, 
Communications in Contemporary Mathematics 2(1), 1--34, 2000

\bibitem{att-mainge-red2012} H. Attouch, P.-E. Maing\'{e}, P. Redont, {\it A second-order differential system with Hessian-driven damping;
application to non-elastic shock laws},  Differential Equations and Applications 4(1), 27--65, 2012

\bibitem{att-peyp-red2016} H. Attouch, J. Peypouquet, P. Redont, {\it Fast convex optimization via inertial dynamics with Hessian driven damping}, 
Journal of Differential Equations 261, 5734--5783, 2016

\bibitem{att-red2001} H. Attouch, P. Redont, {\it The second-order in time continuous Newton method},
in Approximation, optimization and mathematical economics (Pointe-à-Pitre, 1999), 25--36, Physica, Heidelberg, 2001

\bibitem{att-sv2011} H. Attouch, B.F. Svaiter, {\it A continuous dynamical Newton-like approach to solving monotone inclusions}, 
SIAM Journal on Control and Optimization 49(2), 574--598, 2011

\bibitem{bb-ts-cont} S. Banert, R.I. Bo\c{t}, {\it A forward-backward-forward differential equation and its asymptotic properties}, 
to appear in Journal of Convex Analysis, arXiv:1503.07728, 2015

\bibitem{bauschke-book} H.H. Bauschke, P.L. Combettes, {\it Convex Analysis and Monotone Operator Theory in Hilbert Spaces}, CMS Books in Mathematics, Springer, New York, 2011

\bibitem{bac-co-lu} H.H. Bauschke, P.L. Combettes, D.R. Luke, {\it 
Phase retrieval, error reduction algorithm, and Fienup variants: a view from convex optimization}, 
Journal of the Optical Society of America A. Optics, Image Science, and Vision 19(7), 1334--1345, 2002

\bibitem{bolte-2003} J. Bolte, {\it Continuous gradient projection method in Hilbert spaces}, Journal of Optimization Theory and its 
Applications 119(2), 235--259, 2003

\bibitem{b-d-l2006} J. Bolte, A. Daniilidis, A. Lewis, {\it The \L{}ojasiewicz inequality for nonsmooth subanalytic functions with applications 
to subgradient dynamical systems}, SIAM Journal on Optimization 17(4), 1205--1223, 2006

\bibitem{b-d-l-s2007} J. Bolte, A. Daniilidis, A. Lewis, M. Shota, {\it Clarke subgradients of stratifiable functions}, 
SIAM Journal on Optimization 18(2), 556--572, 2007

\bibitem{b-d-l-m2010} J. Bolte, A. Daniilidis, O. Ley, L. Mazet, {\it Characterizations of \L{}ojasiewicz inequalities:
subgradient flows, talweg, convexity}, Transactions of the American Mathematical Society 362(6), 3319--3363, 2010

\bibitem{b-sab-teb} J. Bolte, S. Sabach, M. Teboulle, {\it Proximal alternating linearized minimization 
for nonconvex and nonsmooth problems}, Mathematical Programming Series A (146)(1--2), 459--494, 2014

\bibitem{borwein-zhu} J.M. Borwein, Q.J. Zhu, {\it Techniques of Variational Analysis}, 
Springer, New York, 2005

\bibitem{b-c-inertial-nonc-ts} R.I. Bo\c t, E.R. Csetnek, 
{\it An inertial Tseng's type proximal algorithm for nonsmooth and nonconvex optimization problems}, 
Journal of Optimization Theory and Applications, DOI 10.1007/s10957-015-0730-z 

\bibitem{b-c-dyn-KM} R.I. Bo\c t, E.R. Csetnek, {\it A dynamical system associated with the fixed points set of a 
nonexpansive operator}, Journal of Dynamics and Differential Equations, DOI: 10.1007/s10884-015-9438-x, 2015

\bibitem{b-c-hess} R.I. Bo\c t, E.R. Csetnek, {\it A second order dynamical system with Hessian-driven damping and penalty term associated to 
variational inequalities}, arXiv:1608.04137, 2016

\bibitem{b-c-dyn-pen} R.I. Bo\c t, E.R. Csetnek, {\it Approaching the solving of constrained variational inequalities via penalty 
term-based dynamical systems}, Journal of Mathematical Analysis and Applications 435(2), 1688--1700, 2016

\bibitem{b-c-dyn-sec-ord} R.I. Bo\c t, E.R. Csetnek, {\it Second order forward-backward dynamical systems for
monotone inclusion problems}, Siam Journal on Control and Optimization 54(3), 1423--1443, 2016

\bibitem{b-c-conv-rate-cont} R.I. Bo\c t, E.R. Csetnek, {\it Convergence rates for forward-backward dynamical systems 
associated with strongly monotone inclusions}, to appear in Journal of Mathematical Analysis and Applications, arXiv:1504.01863, 2015

\bibitem{bcl} R.I. Bo\c t, E.R. Csetnek, S. L\'aszl\'o, {\it An inertial forward-backward algorithm for the minimization of the sum of 
two nonconvex functions}, EURO Journal on Computational Optimization 4, 3--25, 2016 

\bibitem{brezis} H. Br\'{e}zis, {\it Op\'{e}rateurs maximaux monotones et semi-groupes de contractions dans les espaces de Hilbert}, 
North-Holland Mathematics Studies No. 5, Notas de Matem\'{a}tica (50), North-Holland/Elsevier, New York, 1973

\bibitem{c-pesquet-r} E. Chouzenoux, J.-C. Pesquet, A. Repetti, {\it Variable metric forward-backward algorithm for minimizing the sum of a 
differentiable function and a convex function}, Journal of Optimization Theory and its Applications 162(1), 107--132, 2014

\bibitem{f-g-peyp} P. Frankel, G. Garrigos, J. Peypouquet, {\it Splitting methods with variable metric for Kurdyka-\L{}ojasiewicz functions and general 
convergence rates}, Journal of Optimization Theory and its Applications 165(3), 874--900, 2015

\bibitem{haraux} A. Haraux, {\it Syst\`{e}mes Dynamiques Dissipatifs et Applications},   
Recherches en Math\'{e}- matiques Appliqu\'{e}ées 17, Masson, Paris, 1991

\bibitem{h-j} A. Haraux, M. Jendoubi, {\it Convergence of solutions of second-order gradient-like systems with analytic nonlinearities}, 
Journal of Differential Equations 144(2), 313--320, 1998

\bibitem{h-l-s-t} R. Hesse, D.R. Luke, S. Sabach, M.K. Tam, {\it Proximal heterogeneous block input-output method and application 
to blind ptychographic diffraction imaging}, SIAM Journal on Imaging Sciences 8(1), 426--457, 2015

\bibitem{kurdyka1998} K. Kurdyka, {\it On gradients of functions definable in o-minimal structures}, 
Annales de l'institut Fourier (Grenoble) 48(3), 769--783, 1998

\bibitem{lojasiewicz1963} S. \L{}ojasiewicz, {\it Une propri\'{e}t\'{e} topologique des sous-ensembles analytiques r\'{e}els}, 
Les \'{E}quations aux D\'{e}riv\'{e}es Partielles, \'{E}ditions du Centre National de la Recherche Scientifique Paris, 87--89, 1963 

\bibitem{boris-carte} B. Mordukhovich, {\it Variational Analysis and Generalized Differentiation, I: Basic Theory, II: Applications}, 
Springer-Verlag, Berlin, 2006

\bibitem{ipiano} P. Ochs, Y. Chen, T. Brox, T. Pock, {\it iPiano: Inertial proximal algorithm for non-convex optimization}, SIAM Journal on 
Imaging Sciences 7(2), 1388--1419, 2014

\bibitem{rock-wets} R.T. Rockafellar, R.J.-B. Wets, {\it Variational Analysis}, Fundamental Principles of Mathematical Sciences 317, 
Springer-Verlag, Berlin, 1998

\bibitem{simon} L. Simon, {\it Asymptotics for a class of nonlinear evolution equations, with applications to geometric problems}, 
Annals of Mathematics (2) 118, 525--571, 1983

\end{thebibliography}
\end{document}